\documentclass{llncs}
\usepackage{amsmath,amssymb}
\usepackage{mathrsfs}
\usepackage{xypic}\xyoption{tips}
\usepackage{isomath}
\newcommand{\rGamma}{\mathrm{\Gamma}}

\newcommand{\ZZ}{\mathbb{Z}}
\newcommand{\PP}{\mathbb{P}}
\newcommand{\cO}{\mathcal{O}}
\newcommand{\oO}{\mathrm{O}}
\renewcommand{\cL}{\mathscr{L}}
\newcommand{\cK}{\mathscr{K}}
\newcommand{\cM}{\mathscr{M}}
\newcommand{\isom}{\cong}
\newcommand{\End}{\mathrm{End}}
\newcommand{\Hom}{\mathrm{Hom}}
\newcommand{\Mat}{\mathrm{M}}
\newcommand{\ignore}[1]{}

\title{Arithmetic of split Kummer surfaces:\\
Montgomery endomorphism\\ of Edwards products}
\author{David Kohel}
\institute{
Institut de Math{\'e}matiques de Luminy\\
Universit{\'e} de la M{\'e}diterran{\'e}e\\
163, avenue de Luminy, Case 907\\
13288 Marseille Cedex 9\\ France}
\date{}

\begin{document}
\maketitle

\begin{abstract}
Let $E$ be an elliptic curve, $\cK_1$ its Kummer curve 
$E/\{\pm1\}$, $E^2$ its square product, and $\cK_2$ the 
split Kummer surface $E^2/\{\pm1\}$.  
The addition law on $E^2$ gives a large endomorphism 
ring, which induce endomorphisms of $\cK_2$. 
With a view to the practical applications to scalar 
multiplication on $\cK_1$, we study the explicit 
arithmetic of $\cK_2$.
\end{abstract}

\section{Introduction}
\label{section:introduction}

Let $A$ be an abelian group, whose group law is expressed additively.
Let $\Mat_2(\ZZ)$ be the subring of $\End(A^2)$, acting as 
$$
\alpha(x,y) = (ax+by, cx+dy) \mbox{ where } 
\alpha = \left(\begin{array}{cc}a&b\\c&d\end{array}\right)\!\cdot
$$ 
Define endomorphisms $\sigma$ and $\varphi_i$ by
$$
\sigma = \left(\begin{array}{r@{\;\;}r}0&1\\1&0\end{array}\right)\!, \ 
\varphi = \varphi_0 
         = \left(\begin{array}{r@{\;\;}r}2&0\\1&1\end{array}\right)\!, \mbox{ and }   
\varphi_1 = \sigma\varphi\sigma = 
  \left(\begin{array}{r@{\;\;}r}1&1\\0&2\end{array}\right)\!.
$$
The Montgomery ladder for scalar multiplication by an integer $n$ is 
expressed on $A^2$ by the recursion 
$$
v_r = (0,x) \mbox{ and } v_i = \varphi_{n_i}(v_{i+1}) \mbox{ for } i = r-1,\dots,1,0,
$$
where $n$ has binary representation $n_{r-1}\dots n_1n_0$.  
The successive steps $v_i$ in the ladder are of the form $(mx,(m+1)x)$ and
$v_0 = (nx,(n+1)x)$, from which we output $nx$ 
(see Montgomery~\cite{Montgomery} and Joye~\cite{JoyeYen} for general formulation).
We refer to $\varphi$ as the {\it Montgomery endomorphism}. 

Since $-1$ is an automorphism in the center of $\Mat_2(\ZZ)$, an endomorphism 
of $A^2$ also acts on the quotient $A^2/\{\pm 1\}$. 
In particular, we will derive expressions of the above operators on the split 
Kummer surface $\cK_2 = E^2/\{\pm 1\}$ associated to an elliptic curve $E$.  

Prior work has focused on Kummer curves $\cK_1 = E/\{\pm 1\} \isom \PP^1$, 
determined by the quotient $\pi: E \rightarrow \cK_1$, often expressed as 
operating only on the $x$-coordinate of a Weierstrass model 
(see Montgomery~\cite{Montgomery}, Brier and Joye~\cite{BrierJoye} and Izu 
and Takagi~\cite{IzuTakagi}). Such methods consider the full quotient 
$\cK_1^2 = E^2/\{(\pm 1,\pm 1)\}$. For this approach one takes the endomorphism 
$$
\rho = \left(\begin{array}{r@{\;\;}r}1&1\\1&-1\end{array}\right)\!,\   
$$ 
arising in duplication formulas for theta functions~\cite{Mumford}. 
This endomorphism satisfies $\rho^2 = 2$, giving a factorization 
of $2$ in $\End(E^2)$, and induces an endomorphism of $\cK_2$, which 
we also refer to as $\rho$.  This gives a commutative diagram:
$$
\SelectTips{cm}{10}
\xymatrix@R=6mm@C=12mm{
 E^2 \ar[d]\ar[r]^\rho\ar[rd] &  E^2 \ar[d] \\
\cK_2 \ar[d]\ar[r]^\rho\ar[rd] & \cK_2 \ar[d] \\
\cK_1^2                   & \cK_1^2.
}
$$
Although $\rho$ does not extend to an endomorphism of $\cK_1^2$ we obtain 
a system of polynomial equations in $\cK_1^2 \times \cK_1^2$ from the graph:
$$
\rGamma_\rho = \left\{ \big((P,Q),\rho(P,Q)\big) : (P,Q) \in E^2 \right\} 
              \subset E^2 \times E^2.
$$
One recovers $\pi(P+Q)$ from specializing this sytem at known points 
$\pi(P)$, $\pi(Q)$ and $\pi(P-Q)$.  By considering the partial quotient 
$\cK_2$ as a double cover of $\cK_1^2$, we obtain endmorphisms of $\cK_2$ 
induced by the isogenies $\rho$ as well as $\varphi_0$ and $\varphi_1$.

Since the structure of addition laws of abelian varieties, or isogenies 
in general, depends intrinsically on the embedding in projective space 
(see~\cite{Kohel}, \cite{LangeRuppert}), we develop specific models for 
the Kummer surface $\cK_2$ associated to a model of an elliptic curve $E$ 
with prescribed embedding.  For this purpose we investigate Edwards models 
for elliptic curves embedded in $\PP^3$. 

\section{Projective embeddings of a Kummer variety $\cK$}
\label{section:Kummer_embeddings}

Let $k$ be a field of characteristic different from $2$ and $A/k$ 
an abelian variety.  An addition law on $A$ is defined by Lange 
and Ruppert~\cite{LangeRuppert} to be a polynomial representative 
for the addition morphism $A^2 \rightarrow A$. Such maps depend  
in an essential way on its projective embedding. 
Similarly, the explicit polynomial maps for morphisms of the 
Kummer variety $\cK = A/\{\pm 1\}$ depend on a choice of its 
projective embedding.  We approach the problem of embedding $\cK$ 
in the following way.  

Let $i: A \rightarrow \PP^r$ be a projectively normal embedding 
(see~\cite{Kohel} for a definition and motivation for this hypothesis),
determined by a symmetric invertible sheaf 
$\cL = \cO_A(1) = i^*\cO_{\PP^r}(1)$ and let $\pi: A \rightarrow \cK$ 
be the projection morphism. 
We say that an embedding $j: \cK \rightarrow \PP^s$ is {\it compatible} 
with $i: A \rightarrow \PP^r$ if $\pi$ is represented by a linear 
polynomial map.  In terms of the invertible sheaf $\cL_1 = \cO_{\cK}(1) 
= j^*\cO_{\PP^s}(1)$, this condition is equivalent to: 
$$
\Hom(\pi^*\cL_1,\cL) \isom \rGamma(A,\pi^*\cL_1^{-1}\otimes\cL) \ne 0,
$$
where $\rGamma(A,\cM)$ is the space of global sections for a sheaf $\cM$.
If we have $\pi^*\cL_1 \isom \cL$ then $\Hom(\pi^*\cL_1,\cL) \isom k$,
and $\pi$ admits a unique linear polynomial map, up to scalar. 

Conversely we can construct an embedding of $\cK$ comptable with given 
$i: A \rightarrow \PP^r$ as follows. The condition that 
$i:A \rightarrow \PP^r$ is projectively normal is equivalent to an 
isomorphism of graded rings 
$$
k[X_0,X_1,\dots,X_r]/I_A = \bigoplus_{n=0}^\infty \rGamma(A,\cL^n),
$$
where $I_A$ is the defining ideal for $A$ in $\PP^r$. We fix an 
isomorphism $\cL \isom [-1]^*\cL$, from which we obtain an eigenspace 
decomposition of the spaces $\rGamma(A,\cL^n)$: 
$$
\rGamma(A,\cL^n) = \rGamma(A,\cL^n)^+ \oplus \rGamma(A,\cL^n)^-.
$$
The sign is noncanonical, but we may choose the sign for the isomorphism 
$\cL \isom [-1]^*\cL$ such that $\dim\rGamma(A,\cL)^+ \ge \dim\rGamma(A,\cL)^-$.  
Setting $V = \rGamma(A,\cL)^+$, we define $j:\cK \rightarrow \PP^s$ by the 
image of $A$ in $\PP^s = \PP(V)$.  
This defines the sheaf $\cL_1 = j^*\cO_{\PP^s}(1)$ and gives a homomorphism 
$\pi^*\cL_1 \rightarrow \cL$.  

In what follows we carry out this construction to determine projective 
embeddings for the Kummer varieties $\cK_1$ and $\cK_2$ associated to 
an elliptic curve embedded as an Edwards model in $\PP^3$, and study the 
form of the endomorphisms $\sigma$, $\varphi$ and $\rho$. 

\section{Edwards model and projective embeddings of $\cK_1$}
\label{section:K1_embedding}

Let $E$ be an elliptic curve embedded in $\PP^3$ as an Edwards model 
(see Edwards~\cite{Edwards}, Bernstein and Lange~\cite{BernsteinLange-Edwards},
and Hisil et al.~\cite{Hisil-EdwardsRevisited} or Kohel~\cite{Kohel} 
for this form):
$$
X_0^2 + d X_3^2 = X_1^2 + X_2^2,\quad X_0 X_3 = X_1 X_2,  
$$
with identity $\oO = (1:0:1:0)$, and negation map 
$$
[-1](X_0:X_1:X_2:X_3) = (X_0:-X_1:X_2:-X_3).
$$ 
The eigenspace decomposition for $\rGamma(E,\cL)$ is 
$$
\rGamma(E,\cL)
  = \bigoplus_{i=1}^4 kX_i 
    = (kX_0 \oplus kX_1) \oplus (kX_2 \oplus kX_3).
$$
The Kummer curve of $E$ is $\cK_1 \isom \PP^1$, with quotient map
$$
(X_0:X_1:X_2:X_3) \mapsto (X_0:X_2) = (X_1:X_3).
$$ 
We can now express the scalar multiplication by 2 on $\cK_1$ in 
terms of coordinate functions $X_0, X_1$ on $\cK_1$.   

\begin{lemma}
The duplication morphism $[2]:\cK_1 \rightarrow \cK_1$ is uniquely 
represented by the polynomial map
$$
(X_0:X_1) \mapsto ( (d-1)X_0^4 - d (X_0^2 - X_1^2)^2 : (X_0^2 - X_1^2)^2 + (d-1) X_1^4 ).
$$
\end{lemma}

\begin{proof}
The correctness of the polynomial map can be directly verified 
by the fact that the known endomorphisms $[2]$ on $E$ commutes 
with $\pi$ and the above polynomial map for $[2]$ on $\cK_1$. 
The uniqueness follows from the existence of the above degree 
four polynomial expressions, since from $\deg([2]) = 4$ we 
obtain $[2]^*\cL_1 \isom \cL_1^4$.
Since degree $n$ polynomial expressions for a morphism $\psi$ 
are in bijection with 
$$
\Hom(\psi^*\cL_1,\cL_1^n) \isom 
  \rGamma(E,\psi^*\cL_1^{-1}\otimes\cL_1^n),
$$
the result follows.
\qed
\end{proof}

\section{Segre embedings and projective products}
\label{section:Segre_embeddings}

In general a projective model behaves well with respect to 
the theory.  In order to characterize a product $X \times Y$ 
with $X \subseteq \PP^r$ and $Y \subseteq \PP^s$ we apply 
the Segre embedding $S: \PP^r \times \PP^s \rightarrow \PP^{rs+r+s}$
given by 
$$
\big( (X_0:X_1:\dots:X_r), (Y_0:Y_1:\dots:Y_s) \big)
\longmapsto (X_0Y_0:X_1Y_0:\dots:X_rY_s),
$$
and consider the image $S(X \times Y)$ in $\PP^{rs+r+s}$.

For $r = s = 1$, we have $(r+1) + (s+1) = 4$ coordinates 
to represent a point in $\PP^1 \times \PP^1$ and 
$(r+1)(s+1) = 4$ coordinates for a point in $\PP^3$. 
For higher degrees or powers 
$\PP^{r_1} \times \cdots \times \PP^{r_t}$ the Segre embedding 
becomes unwieldy for explicit computation.

In particular, for the product $\cK_1^2 \isom \PP^1 \times 
\PP^1$ this gives the embedding of $\cK_1^2$ in $\PP^3$ 
as the hypersurface $U_0U_3 = U_1U_2$, given by  
$$
\big( (X_0:X_1), (Y_0:Y_1) \big) \longmapsto 
(U_0:U_1:U_2:U_3) = (X_0Y_0:X_1Y_0:X_0Y_1:X_1Y_1).
$$
The inverse is given by the product of projections 
$\pi_1: S(\cK_1^2) \rightarrow \cK_1$
$$
(U_0:U_1:U_2:U_3) \longmapsto (U_0:U_1) = (U_2:U_3),
$$
and $\pi_2: S(\cK_1^2) \rightarrow \cK_1$
$$
(U_0:U_1:U_2:U_3) \longmapsto (U_0:U_2) = (U_1:U_3).
$$
Each projection is represented locally by a two-dimensional 
space of linear polynomial maps, but no such map defines 
$\pi_i$ globally as a morphism. 

We use the Segre embedding $\cK_1^2 \rightarrow S(\cK_1^2)$ 
to provide a projective embedding for $\cK_1^2$ and construct 
$\cK_2$ as a double cover of $S(\cK_1^2)$ in 
$S(\cK_1^2) \times \PP^1 \subseteq \PP^3 \times \PP^1$. 
To preserve the compactness of the representation we work with 
the model in $\PP^3 \times \PP^1$, rather than its model in 
$\PP^7$, however we give this model in Theorem~\ref{K2_models}.

In order to define a morphism $\cK_2 \rightarrow \cK_2$ it 
suffices to make use of the factorization through $\cK_1^2 
\times \PP^1$ to each of the products.  Thus a morphism 
$\psi: X \rightarrow \cK_2$ is determined by three maps 
$\psi_i = \pi_i\circ\psi$ for $1 \le i \le 3$, and a composition 
with a Segre embedding of $\cK_1^2$ to $\PP^3$ gives the 
map to $\cK_2$ in $\PP^3 \times \PP^1$.  We note, however,
that expansion of polynomial maps for this factorization 
$S \circ (\pi_1 \times \pi_2)$ may yield polynomial maps 
of higher degree than $\cK_2 \rightarrow S(\cK_1^2)$ 
directly (see Theorem~\ref{K2_models}).
\vspace{2mm}

\noindent{\bf Note.}
Despite the isomorphism $\cK_1 \isom \PP^1$, and even equality 
under the projective embedding, we write $\cK_1^2$ and $\cK_1^2 
\times \PP^1$ rather than $(\PP^1)^2$ and $(\PP^1)^3$ in order 
to reflect the distinguished role of the two Kummer curves 
in this product. 

\section{Edwards model and projective embeddings of $\cK_2$}
\label{section:K2_embeddings}

We now describe the embeddings of $\cK_2$ as a double cover of $\cK_1^2$. 

\begin{theorem}
\label{K2_models}
Let $E: X_0^2 + d X_3^2 = X_1^2 + X_2^2, X_0X_3 = X_1X_3$ be 
an elliptic curve in $\PP^3$ with identity $\oO = (1:0:1:0)$. 
The Kummer surface $\cK_2$ has a model as a hypersurface in 
$\cK_1^2 \times \PP^1$ given by 
$$
(X_0^2-X_1^2)(Y_0^2-Y_1^2)Z_0^2 = (X_0^2-dX_1^2)(Y_0^2-dY_1^2)Z_1^2,
$$
with base point $\pi(\oO) = ((1:1),(1:1),(1:0))$, and 
projection $E^2 \rightarrow \cK_2$ given by 
$\pi_1(P,Q) = (X_0:X_2)$, $\pi_2(P,Q) = (Y_0:Y_2)$, and 
$$
\pi_3(P,Q) = 
(X_0Y_0:X_1Y_1) = 
(X_2Y_0:X_3Y_1) = 
(X_0Y_2:X_1Y_3) = 
(X_2Y_2:X_3Y_3),
$$
where $(P,Q) = \big((X_0:X_1:X_2:X_3), (Y_0:Y_1:Y_2:Y_3)\big)$.

Under the Segre embedding $S: \cK_1^2 \mapsto \PP^3$, this 
determines the variety in $\PP^3 \times \PP^1$ cut out by 
$$
(U_0^2 - U_1^2 - U_2^2 + U_3^2) Z_0^2 
= (U_0^2 - d U_1^2 - d U_2^2 + d^2 U_3^2) Z_1^2, 
$$
on the hypersurface $U_0 U_3 = U_1 U_2$ defining $S(\cK_1^2)$.
The Segre embedding of $\cK_2$ in $\PP^7$ is cut 
out by the quadratic relation
$$
T_0^2 - T_1^2 - T_2^2 + T_3^2 = T_4^2 - d T_5^2 - d T_6^2 + d^2 T_7^2,
$$
on the image of the Segre embedding of $(\PP^1)^3 \rightarrow \PP^7$, 
determined by:
$$
\begin{array}{ccc}
T_0 T_3 = T_1 T_2,\ T_0 T_5 = T_1 T_4,\ T_0 T_6 = T_2 T_4,\\ 
T_0 T_7 = T_3 T_4,\ T_1 T_6 = T_3 T_4,\ T_1 T_7 = T_3 T_5,\\ 
T_2 T_5 = T_3 T_4,\ T_2 T_7 = T_3 T_6,\ T_4 T_7 = T_5 T_6.
\end{array}
$$
The morphism to $E^2 \rightarrow S(\cK_2) \subseteq \PP^7$ 
is determined by:
$$
(
  X_0 Y_0 : X_2 Y_0 : X_0 Y_2 : X_2 Y_2, 
  X_1 Y_1 : X_3 Y_1 : X_1 Y_3 : X_3 Y_3
).
$$
\end{theorem}

\begin{proof}
The quadratic relation for $\cK_2$ in $\cK_1^2 \times \PP^1$: 
$$
(X_0^2-X_1^2)(Y_0^2-Y_1^2)Z_0^2 = (X_0^2-dX_1^2)(Y_0^2-dY_1^2)Z_1^2,
$$
follows by pulling back the relation to $E^2$ by  
$$
\pi^*(Y_1/Y_0) = (Y_2/Y_0),\  
\pi^*(X_1/X_0) = (X_2/X_0),\  
\pi^*(Z_1/Z_0)^2 = (X_1Y_1/X_0Y_0)^2.
$$ 
Since the morphism maps through $E^2/\{\pm 1\}$, defines a 
double cover of $\cK_1^2$, and is irreducible, we conclude 
that the quadratic relation determines $\cK_2$. 
   The remaining models follow by tracing this quadratic 
relation through the Segre embeddings. 

The last model, in $\PP^7$, can be interpreted as coming from the 
construction of Section~\ref{section:Kummer_embeddings}, applied 
to the Segre embedding of $E^2$ in $\PP^{15}$. 
The sixteen-dimensional space of global sections splits into two 
eight-dimensional subspaces, for which 
$$
\{
  X_0 Y_0 : X_2 Y_0 : X_0 Y_2 : X_2 Y_2, 
  X_1 Y_1 : X_3 Y_1 : X_1 Y_3 : X_3 Y_3
\}
$$
forms a basis for the plus one eigenspace. The compatibility of 
the maps from $E^2$ is verified by projecting from the models 
in $\PP^7$ and $\PP^3 \times \PP^1$ to $\cK_1^2 \times \PP^1$.  
\qed
\end{proof}

The description of the maps in the previous theorem, together with the 
action of $[-1]$ on the Edwards model, implies the next corollary.

\begin{corollary}
\label{corollary:K2_automorphisms}
The automorphism $\sigma: E^2 \rightarrow E^2$ given by $(P,Q) \mapsto 
(Q,P)$ induces the automorphism of $\cK_2$ in the respective models 
in $\cK_1^2 \times \PP^1$, $\PP^3 \times \PP^1$ and $\PP^7$:
$$
\big((X_0:X_1),(Y_0:Y_1),(Z_0:Z_1)\big) \mapsto \big((Y_0:Y_1),(X_0:X_1),(Z_0:Z_1)\big),
$$
$$
\big((U_0:U_1:U_2:U_3),(Z_0:Z_1)\big) \mapsto \big((U_0:U_2:U_1:U_3),(Z_0:Z_1)\big),
$$
$$
(T_0:T_1:T_2:T_3:T_4:T_5:T_6:T_7) \mapsto (T_0:T_2:T_1:T_3:T_4:T_6:T_5:T_7).
$$
The automorphism $\iota: \cK_2 \rightarrow \cK_2$ induced by the 
automorphisms $[-1]\times[1]$ and $[1] \times [-1]$ of $E^2$ is given by:
$$
\big((X_0:X_1),(Y_0:Y_1),(Z_0:Z_1)\big) \mapsto \big((X_0:X_1),(Y_0:Y_1),(Z_0:-Z_1)\big),
$$
$$
\big((U_0:U_1:U_2:U_3),(Z_0:Z_1)\big) \mapsto \big((U_0:U_1:U_2:U_3),(Z_0:-Z_1)\big),
$$
$$
(T_0:T_1:T_2:T_3:T_4:T_5:T_6:T_7) \mapsto (T_0:T_1:T_2:T_3:-T_4:-T_5:-T_6:-T_7).
$$
\end{corollary}

\section{Endomorphisms of Kummer surfaces $\cK_2$}
\label{section:K2_endomorphisms}

We are now able to define polynomial maps for the Montgomery endomorphism 
$\varphi$, where $\rho$, $\tau$, and $\varphi$ are the endomorphisms 
$$
\varphi = \left(\begin{array}{r@{\;\;}r}2&0\\1&1\end{array}\right) \mbox{ and }
\rho = \left(\begin{array}{r@{\;\;}r}1&1\\1&-1\end{array}\right) \mbox{ and }
\tau = \left(\begin{array}{r@{\;\;}r}1&1\\0&1\end{array}\right)\!,
$$
as elements of $\Mat_2(\ZZ)/\{\pm 1\}$. In addition we recall the definitions 
$$
\iota = \left(\begin{array}{r@{\;\;}r}-1&0\\0&1\end{array}\right) = 
        \left(\begin{array}{r@{\;\;}r}1&0\\0&-1\end{array}\right) \mbox{ and }
\sigma = \left(\begin{array}{r@{\;\;}r}0&1\\1&0\end{array}\right)\!,
$$
and note the commuting relations $\rho \circ \iota = \sigma \circ \rho$ and 
$\rho \circ \sigma = \iota \circ \rho$ for $\iota$, $\sigma$, and $\rho$.

Explicit polynomial maps for the Montgomery endomorphism $\varphi$ 
on $\cK_2$ follow from the identities
$$
\varphi_0 = \varphi = \tau \circ \sigma \circ \rho \mbox{ and } 
\varphi_1 = \sigma \circ \varphi \circ \sigma.
$$ 
As a consequence the Montgomery ladder can be expressed in terms of the 
automorphisms $\sigma$, $\iota$, and endomorphisms $\rho$ and $\tau$. 
The following two theorems, whose proof follows from standard addition 
laws on the Edwards model (see Bernstein and 
Lange~\cite{BernsteinLange-Edwards},~\cite{BernsteinLange-EdwardsComplete}, 
Hisil~\cite{Hisil-EdwardsRevisited}, and Kohel~\cite{Kohel}), and 
verification of the commutativity relations $\pi \circ \psi = \psi \circ \pi$
for an endomorphism $\psi$.  

\begin{theorem}
\label{theorem:K2_rho}
The projections of the endomorphisms $\rho: \cK_2 \rightarrow \cK_2$ 
are uniquely represented by polynomials of bidegree $(1,1)$, $(1,1)$,
and $(2,0)$, explicitly:
$$
\begin{array}{l}
\pi_1\circ\rho\big((U_0:U_1:U_2:U_3),(Z_0:Z_1)\big) =
( U_0 Z_0 - d U_3 Z_1 : -U_0 Z_1 + U_3 Z_0 )\\
\pi_2\circ\rho\big((U_0:U_1:U_2:U_3),(Z_0:Z_1)\big) =
(U_0 Z_0 + d U_3 Z_1 : U_0 Z_1 + U_3 Z_0),\\ 
\pi_3\circ\rho\big((U_0:U_1:U_2:U_3),(Z_0:Z_1)\big) =
(U_0^2 - d U_3^2 \,:\, -U_1^2 + U_2^2).
\end{array}
$$
The projection $\rho:\cK_2\rightarrow S(\cK_1^2)$ admits a two-dimensional 
space of polynomial maps of bidegree $(2,1)$ spanned by:
$$
\begin{array}{l}
\begin{array}{rl}
\big( 
  & U_0^2 - d U_1^2 - d U_2^2 + d U_3^2) Z_0 \,:\, \\
  & -(d-1) U_0 U_3 Z_0 - (U_0^2 - d U_1^2 - d U_2^2 + d^2 U_3^2) Z_1 \,:\, \\
  & -(d-1) U_0 U_3 Z_0 + (U_0^2 - d U_1^2 - d U_2^2 + d^2 U_3^2) Z_1 \,:\, \\
  & -(U_0^2 - U_1^2 - U_2^2 + d U_3^2) Z_0\ \big)
\end{array}\\
\begin{array}{rl}
\big( 
 & (U_0^2 - d U_1^2 - d U_2^2 + d U_3^2) Z_1 \,:\, \\
 & -(d - 1) U_0 U_3 Z_1 - (U_0^2 - U_1^2 - U_2^2 + U_3^2) Z_0 \,:\, \\
 & -(d - 1) U_0 U_3 Z_1 + (U_0^2 - U_1^2 - U_2^2 + U_3^2) Z_0 \,:\, \\
 & -(U_0^2 - U_1^2 - U_2^2 + d U_3^2) Z_1 \big).
\end{array}
\end{array}
$$
\end{theorem}

\begin{theorem}
\label{theorem:K2_tau}
The maps $\pi_i\circ\tau: \cK_2 \rightarrow \cK_1$ are given by 
$$
\begin{array}{l}
\pi_1\circ\tau\big((U_0:U_1:U_2:U_3),(Z_0:Z_1)\big) 
  = (U_0 Z_0 - d U_3 Z_1 : -U_0 Z_1 + U_3 Z_0),\\
\pi_2\circ\tau\big((U_0:U_1:U_2:U_3),(Z_0:Z_1)\big) 
  = (U_0 : U_2) = (U_1 : U_3).
\end{array}
$$
and $\pi_3\circ\tau\big((U_0:U_1:U_2:U_3),(Z_0:Z_1)\big)$ is given 
by the equivalent expressions 
$$
\begin{array}{l}
\big( (U_0^2 - d U_3^2) Z_0 \,:\, (U_0 U_1 - U_2 U_3) Z_0 + (U_0 U_2 - d U_1 U_3) Z_1\big)\\
\big( -(U_0 U_2 - U_1 U_3) Z_0 + (U_0 U_1 - d U_2 U_3) Z_1 \,:\, (U_1^2 - U_2^2) Z_1\big)
\end{array}
$$
\ignore{
The composition $\cK_2 \rightarrow S(\cK_1^2) \subseteq \PP^3$ can be 
expressed in lower degree polynomials directly:
$$
\begin{array}{l}
\big( U_0 (U_0 Z_0 - d U_3 Z_1) \,:\, 
      U_0 (U_3 Z_0 - U_0 Z_1) \,:\, 
      U_2 (U_0 Z_0 - d U_3 Z_1) \,:\, 
      U_2 (U_3 Z_0 - U_0 Z_1) \big)\\
= 
\big( U_1 (U_0 Z_0 - d U_3 Z_1) \,:\, 
      U_1 (U_3 Z_0 - U_0 Z_1) \,:\, 
      U_3 (U_0 Z_0 - d U_3 Z_1) \,:\, 
      U_3 (U_3 Z_0 - U_0 Z_1) \big).
\end{array}
$$
}
\end{theorem}

\section{Conclusion}

The above polynomial maps for Montgomery endomorphism $\varphi$ of $\cK_2$ 
allows one to carry out a simultaneous symmetric addition and doubling on 
the Kummer surface.  Besides the potential efficiency of this computation, 
this provides a simple geometric description of the basic ingredient for 
the Montgomery ladder on an Edwards model of an elliptic curve. 
The symmetry of the derived model for the split Kummer surface, and the 
endomorphisms $\iota$, $\sigma$, and $\rho$ provide the tools necessary 
for scalar multiplication on Edwards curves in cryptographic applications 
requiring protection from side channel attacks.


\begin{thebibliography}{99}

\bibitem{BernsteinLange-Edwards}
D.~J.~Bernstein, T.~Lange. 
Faster addition and doubling on elliptic curves.
{\it Advances in Cryptology: ASIACRYPT 2007}, 
Lecture Notes in Computer Science, {\bf 4833}, Springer, 29--50, 2007. 

\bibitem{BernsteinLange-EdwardsComplete}
D.~J.~Bernstein and T.~Lange.
A complete set of addition laws for incomplete Edwards curves, preprint, 
\url{http://eprint.iacr.org/2009/580}, 2009. 

\bibitem{BrierJoye}
E.~Brier and M.~Joye, 
Weierstrass elliptic curves and side-channel attacks,
{\it Public Key Cryptography}, 
{\it Lecture Notes in Comput.~Sci.}, {\bf 2274}, 335--345, 2002.

\bibitem{Edwards}
H.~Edwards.
A normal form for elliptic curves.
Bulletin of the American Mathematical Society, {\bf 44}, 393--422, 2007.

\bibitem{Hisil-EdwardsRevisited}
H.~Hisil, K.~K.-H.~Wong, G.~Carter, E.~Dawson, 
Twisted Edwards curves revisited, 
{\it Advances in cryptology -- ASIACRYPT 2008},
Lecture Notes in Computer Science, {\bf 5350}, Springer, Berlin, 326--343, 2008. 

\bibitem{Kohel}
D.~Kohel.
Addition law structure of elliptic curves.
to appear in {\it Journal of Number Theory}, 
\url{http://arxiv.org/abs/1005.3623}, 2011.

\bibitem{IzuTakagi}
A Fast Parallel Elliptic Curve Multiplication Resistant against Side Channel Attacks,
{\it Public Key Cryptography}, 
{\it Lecture Notes in Comput.~Sci.}, {\bf 2274}, 280--296, 2002.


\bibitem{JoyeYen}
M.~Joye and S.-M.~Yen.
The Montgomery Powering Ladder, 
CHES 2002, {\it Lecture Notes Comp. Sci.}, {\bf 2523}, 291--302, 2003.

\bibitem{LangeRuppert}
H.~Lange and W.~Ruppert.
Complete systems of addition laws on abelian varieties.
{\it Invent. Math.}, {\bf 79} (3), 603--610, 1985.


\bibitem{Montgomery}
P.~Montgomery.
Speeding the {P}ollard and elliptic curve methods of factorization,
{\it Math.~Comp.}, {\bf 48}, no. 177, 243--264, 1987.

\bibitem{Mumford}
D.~Mumford. 
On the equations defining abelian varieties I, 
{\it Invent.~Math.}, {\bf 1}, 287--354, 2966.

\end{thebibliography}
\end{document}